\documentclass[a4, 12pt]{amsart}
\usepackage[mathscr]{eucal}
\usepackage{amssymb}
\usepackage{latexsym}
\usepackage{amsthm}
\usepackage{color}
\theoremstyle{plain}
\newtheorem{theorem}{Theorem}[section]

\newtheorem{remark}{Remark}[section]
\newtheorem{lemma}{Lemma}[section]

\makeatletter
\@addtoreset{equation}{section}

\title[Complete space-like self-expanders in the Minkovski space]
{Complete space-like self-expanders in the Minkovski space}

\author [Z. Li and G. Wei]{Zhi Li and Guoxin Wei}

\address{Zhi Li \\  \newline \indent College of Mathematics and Information Science, Henan Normal University,
\newline \indent 453007, Xinxiang, Henan, China.}
\email{lizhihnsd@126.com}

\address{Guoxin Wei \\ \newline \indent School of Mathematical Sciences, South China Normal University,
\newline \indent 510631, Guangzhou,  China.}
\email{weiguoxin@tsinghua.org.cn}

\begin{document}
\maketitle

\begin{abstract}
It is our purpose to study complete space-like self-expanders in the Minkovski space.
By use of maximum principle of Omori-Yau type, we
can obtain the rigidity theorems on $n$-dimensional complete space-like self-expanders in the Minkovski space $\mathbb R^{n+1}_{1}$. For complete space-like self-expanders of dimension $2$, we give a classification of them under assumption of constant squared norm of the second fundamental form.
\end{abstract}


\section{introduction}
\vskip2mm
\noindent

An $n$-dimensional smooth immersed hypersurface $x: M^{n}\to \mathbb{R}^{n+1}_{1}$ is called a self-expanders of mean curvature flow if its mean curvature vector $\vec{H}$ satisfies the equation
\begin{equation}\label{1.1-1}
\vec{H}=x^{\perp},
\end{equation}
where $\vec{H}=H e_{n+1}$, $x$ and $e_{n+1}$ denote the position vector and the inward unit normal field, respectively.
If we denote by $\langle\cdot, \cdot\rangle$ the standard inner product on $\mathbb{R}^{n+1}_{1}$, then the equation \eqref{1.1-1} is equivalent to
\begin{equation}\label{1.1-2}
H=-\langle x, e_{n+1}\rangle, \ \ \langle e_{n+1}, e_{n+1}\rangle=-1.
\end{equation}

Self-expanders are self similar solutions of mean curvature flow,
that is, the family of hypersurfaces $x_{t}=\sqrt{2t}x, \ \ t>0$ satisfying the equation of mean curvature flow.
Self-expanders are important as they model the behavior of a mean curvature flow
coming out of a conical singularity (\cite{AIC}), and also model the long time behaviors of the flows starting from entire graphs (\cite{Guo}). There are many classification results about the self-expanders.
Ishimura \cite{Ish} and Halldorsson \cite{H} have classified self-expander curves in $\mathbb R^{2}$. In \cite{H}, Halldorsson states that each of the complete self-expander curves immersed in $\mathbb R^{2}$ is convex, properly embedded and asymptotic to the boundary of a cone with vertex at the origin. In 2018, Cheng and Zhou \cite{CZ} studied some properties of complete properly immersed self-expanders. Recently, Ancari and Cheng \cite{AC} studied immersed self-expander hypersurfaces whose mean curvatures have some linear growth controls and obtained some rigidity property of hyperplanes as self-expanders in Euclidean space. More other results about the self-expanders
have been done (cf. \cite{BW}, \cite{EH}, \cite{D}, \cite{FM}, \cite{Smo}, \cite{Sta}, \cite{XY}).

In this paper, by use of maximum principle of Omori-Yau type,
we obtain the classification of complete space-like self-expanders in the Minkowski space $\mathbb R^{n+1}_{1}$.

\begin{theorem}\label{theorem 1.1}
Let $x:M^{n}\rightarrow \mathbb{R}^{n+1}_{1}$
be a complete space-like self-expander with the constant squared norm $S$ of the second fundamental form. If $\inf H^{2}>0$, then $S=1$ and $x(M^{n})$ is the hyperbolic space $\mathbb{H}^{n}(\sqrt{n})$ or the hyperbolic cylinder $\mathbb{H}^{k}(\sqrt{k})\times\mathbb{R}^{n-k}, \ \ 1\leq k\leq n-1$, where $h_{ij}$ are components of the second fundamental form, $S=\sum_{i,j}\limits (h_{ij})^2$.
\end{theorem}

When $n=2$, we completely classify  complete self-expanders with constant squared norm of the second fundamental form in the Minkovski space $\mathbb R^{3}_{1}$.
The study is as follows:

\begin{theorem}\label{theorem 1.2}
Let $x: M^{2}\to \mathbb{R}^{3}_{1}$ be a
$2$-dimensional complete space-like self-expander in $\mathbb R^{3}_{1}$.
If the squared norm $S$ of the second fundamental form is constant, then $x: M^{2}\to \mathbb{R}^{3}_{1}$ is one of the space-like affine plane
$\mathbb {R}^{2}$ through the origin, the hyperbolic
cylinder $\mathbb{H}^{1}(1)\times\mathbb {R}^{1}$ or the hyperbolic space $\mathbb{H}^{2}(\sqrt{2})$, where $h_{ij}$ are components of the second fundamental form, $S=\sum_{i,j}\limits (h_{ij})^2$.
\end{theorem}

\begin{remark}
In fact, using the proof method of Theorem \ref{theorem 1.1}, we can also give the classification of complete self-expanders in the Euclidean space $\mathbb{R}^{3}$ under the condition of the constant squared norm of the second fundamental form, and the classification result is only plane $\mathbb {R}^{2}$ through the origin,
which indicates that the conclusion of Ancari and Cheng \cite{AC} still holds
without assumption on the scalar curvature.
\end{remark}

\vskip5mm
\section {Preliminaries}
\vskip2mm

\noindent

Let $x: M^{n} \rightarrow\mathbb{R}^{n+1}_{1}$ be an
$n$-dimensional space-like hypersurface of $(n+1)$-dimensional  Minkovski space $\mathbb{R}^{n+1}_{1}$. We choose a local orthonormal frame field
$\{e_{A}\}_{A=1}^{n+1}$ in $\mathbb{R}^{n+1}_{1}$ with dual coframe field
$\{\omega_{A}\}_{A=1}^{n+1}$, such that, restricted to $M^{n}$,
$e_{1},\cdots, e_{n}$  are tangent to $M^{n}$. In this paper, we shall make use of the following conventions on the ranges of indices,
$$
 1\leq i,j,k,l\leq n.
$$
Then we have  the structure equations
\begin{equation*}
dx=\sum_i\limits \omega_{i} e_{i}, \ \ de_{i}=\sum_{j}\limits \omega_{ij}e_{j}+\omega_{i n+1}e_{n+1},
\end{equation*}
\begin{equation*}
de_{n+1}=\omega_{n+1 i}e_{i}, \ \ \omega_{n+1i}=\omega_{in+1}
\end{equation*}
where $\omega_{ij}$ is the Levi-Civita connection of the hypersurface.

\noindent By  restricting  these forms to $M$,  we get
\begin{equation}\label{2.1-1}
\omega_{n+1}=0.
\end{equation}

\noindent Taking exterior derivatives of \eqref{2.1-1}, we obtain
\begin{equation}\label{2.1-2}
\omega_{in+1}=\sum_{j} h_{ij}\omega_{j},\quad
h_{ij}=h_{ji}.
\end{equation}

$$
h=\sum_{i,j}h_{ij}\omega_i\otimes\omega_j, \ \ H= \sum_i\limits h_{ii}
$$
are called the second fundamental form and the mean curvature of $x: M\rightarrow\mathbb{R}^{n+1}_{1}$, respectively.
Let $S=\sum_{i,j}\limits (h_{ij})^2$ be  the squared norm
of the second fundamental form  of $x: M\rightarrow\mathbb{R}^{n+1}_{1}$.
The Gauss equations are given by
\begin{equation}\label{2.1-3}
R_{ijkl}=-(h_{ik}h_{jl}-h_{il}h_{jk}).
\end{equation}

\noindent
Defining the
covariant derivative of $h_{ij}$ by
\begin{equation}\label{2.1-4}
\sum_{k}h_{ijk}\omega_k=dh_{ij}+\sum_kh_{ik}\omega_{kj}
+\sum_k h_{kj}\omega_{ki},
\end{equation}
we obtain the Codazzi equations
\begin{equation}\label{2.1-5}
h_{ijk}=h_{ikj}.
\end{equation}
By taking exterior differentiation of \eqref{2.1-4}, and
defining
\begin{equation}\label{2.1-6}
\sum_lh_{ijkl}\omega_l=dh_{ijk}+\sum_lh_{ljk}\omega_{li}
+\sum_lh_{ilk}\omega_{lj}+\sum_l h_{ijl}\omega_{lk},
\end{equation}
we have the following Ricci identities
\begin{equation}\label{2.1-7}
h_{ijkl}-h_{ijlk}=\sum_m
h_{mj}R_{mikl}+\sum_m h_{im}R_{mjkl}.
\end{equation}
Defining
\begin{equation}\label{2.1-8}
\begin{aligned}
\sum_mh_{ijklm}\omega_m&=dh_{ijkl}+\sum_mh_{mjkl}\omega_{mi}
+\sum_mh_{imkl}\omega_{mj}+\sum_mh_{ijml}\omega_{mk}\\
&\ \ +\sum_mh_{ijkm}\omega_{ml}
\end{aligned}
\end{equation}
and taking exterior differentiation of \eqref{2.1-6}, we get
\begin{equation}\label{2.1-9}
\begin{aligned}
h_{ijklq}-h_{ijkql}&=\sum_{m} h_{mjk}R_{milq}
+ \sum_{m}h_{imk}R_{mjlq}+ \sum_{m}h_{ijm}R_{mklq}.
\end{aligned}
\end{equation}

Let $V$ be a tangent $C^{1}$-vector field on $M^{n}$ and denote by $Ric_{V} := Ric-\frac{1}{2}L_{V}g$ the
Bakry-Emery Ricci tensor with $L_{V}$ to be the Lie derivative along the vector field $V$. Define a
differential operator
\begin{equation*}
\mathcal{L}f=\Delta f+\langle V,\nabla f\rangle,
\end{equation*}
where $\Delta$ and $\nabla$ denote the Laplacian and the gradient
operator, respectively. The following maximum principle
of Omori-Yau type concerning the operator $\mathcal{L}$ will
be used in this paper, which was proved by Chen and Qiu \cite{CQ}.

\begin{lemma}\label{lemma 2.1}
Let $(M^{n}, g)$ be a complete Riemannian manifold, and $V$ is a $C^{1}$ vector field on $M^{n}$. If the Bakry-Emery Ricci tensor $Ric_{V}$ is bounded from below, then for any $f\in C^{2}(M^{n})$ bounded from above, there exists a sequence $\{p_{t}\} \subset M^{n}$, such that
\begin{equation*}
\lim_{m\rightarrow\infty} f(p_{t})=\sup f,\quad
\lim_{m\rightarrow\infty} |\nabla f|(p_{t})=0,\quad
\lim_{m\rightarrow\infty}\mathcal{L}f(p_{t})\leq 0.
\end{equation*}
\end{lemma}

We next suppose that $x: M\rightarrow\mathbb{R}^{n+1}_{1}$
is a self-expander, that is, $H=-\langle x, e_{n+1}\rangle$. By a simple calculation, we have the following basic formulas.

\begin{equation}\label{2.1-10}
\aligned
\nabla_{i}H
=-\sum_{k}h_{ik}\langle x, e_{k}\rangle, \ \
\nabla_{j}\nabla_{i}H
=-\sum_{k}h_{ijk}\langle x, e_{k}\rangle-h_{ij}+H\sum_{k}h_{ik}h_{kj}.
\endaligned
\end{equation}

In this paper, we shall always take $V=x^{\top}$ for any formulas.
Using \eqref{2.1-10} and the Ricci identities, we can get the following lemmas.

\begin{lemma}\label{lemma 2.2}
Let $x:M^{n}\rightarrow \mathbb{R}^{n+1}_{1}$ be an $n$-dimensional complete space-like self-expander in $\mathbb R^{n+1}_{1}$. We have
\begin{equation}\label{2.1-11}
\mathcal{L}H=H(S-1),
\end{equation}
\begin{equation}\label{2.1-12}
\aligned
\frac{1}{2}\mathcal{L}
H^{2}=|\nabla H|^{2}+H^{2}(S-1),
\endaligned
\end{equation}
\begin{equation}\label{2.1-13}
\frac{1}{2}\mathcal{L}S
=\sum_{i,j,k}h_{ijk}^{2}+S(S-1).
\end{equation}
\end{lemma}

\begin{lemma}\label{lemma 2.3}
Let $x:M^{n}\rightarrow \mathbb{R}^{n+1}_{1}$ be an $n$-dimensional complete space-like self-expander in $\mathbb R^{n+1}_{1}$. If $S$ is constant, we have
\begin{equation}\label{2.1-14}
\aligned
\sum_{i,j,k,l}(h_{ijkl})^{2}
=&(2-S)\sum_{i,j,k}(h_{ijk})^{2}+6\sum_{i,j,k,l,p}h_{ijk}h_{il}h_{jp}h_{klp}\\
&-3\sum_{i,j,k,l,p}h_{ijk}h_{ijl}h_{kp}h_{lp}.
\endaligned
\end{equation}
\end{lemma}

In order to use the maximum principle of Omori-Yau type, we need the following lemma.
\begin{lemma}\label{lemma 2.4}
Let $x:M^{n}\rightarrow \mathbb{R}^{n+1}_{1}$ be an $n$-dimensional complete space-like self-expander in $\mathbb{R}^{n+1}_{1}$. If the squared norm $S$ of the second fundamental form is constant, the Bakry-Emery Ricci tensor $Ric_{V}$ is bounded from below, where $V=x^{\top}$.
\end{lemma}
\begin{proof}
For any unit vector $e\in T M^{n}$, we can
choose a local tangent orthonormal frame field $\{{e_{i}}\}^{n}_{i=1}$ such that $e=e_{i}$. Then by the definition of self-expander, direct computation gives
\begin{equation*}
\begin{aligned}
\frac{1}{2}L_{x^{\top}}g(e_{i},e_{i})
=&\frac{1}{2}x^{\top}(g(e_{i},e_{i}))-g([x^{\top},e_{i}],e_{i})\\
=&g(\nabla_{e_{i}}(x-x^{\bot}),e_{i})\\
=&1-Hg(\nabla_{e_{i}}e_{n+1},e_{i})\\
=&1-Hh_{ii}.
\end{aligned}
\end{equation*}
Then \eqref{2.1-3} yields
\begin{equation*}
\begin{aligned}
Ric_{x^{\top}}(e_{i},e_{i})
=&Ric(e_{i},e_{i})-\frac{1}{2}L_{x^{\top}}g(e_{i},e_{i})\\
=&-(Hh_{ii}-\sum_{j}h^{2}_{ij})-(1-Hh_{ii})\\
=&\sum_{j}h^{2}_{ij}-1\geq-1.
\end{aligned}
\end{equation*}
The proof of the Lemma \ref{lemma 2.4} is finished.
\end{proof}

\vskip5mm
\section{Proof of Main Theorems}

For an arbitrary fixed point $p$ in $M^{n}$, we always choose a local frame field
$\{e_{i} \}_{i=1}^{n}$
such that
$$
h_{ij}=\lambda_i\delta_{ij}, \ \ S=\sum_{i}\lambda_{i}^{2}, \ \ H^{2}\leq nS,
$$
and the equality of the above inequality holds if and only if
$\lambda_{1}=\lambda_{2}=\cdots=\lambda_{n}$.
Taking exterior derivative of $S=\sum_{i,j}\limits (h_{ij})^2$, we know that
\begin{equation}\label{3.1-1}
\sum_{i,j}h_{ij}h_{ijk}=0, \ \
\sum_{i,j}h_{ij}h_{ijkl}+\sum_{i,j}h_{ijk}h_{ijl}=0, \ \ k,l=1,\cdots,n.
\end{equation}
Besides, from Ricci identities \eqref{2.1-7}, we obtain
\begin{equation*}
h_{ijkl}-h_{ijlk}=
-(\lambda_{i}\lambda_{j}\lambda_{k}\delta_{il}\delta_{jk}
-\lambda_{i}\lambda_{j}\lambda_{l}\delta_{ik}\delta_{jl}
+\lambda_{i}\lambda_{j}\lambda_{k}\delta_{ik}\delta_{jl}
-\lambda_{i}\lambda_{j}\lambda_{l}\delta_{il}\delta_{jk}).
\end{equation*}
That is, \begin{equation}\label{3.1-2}
h_{iijj}-h_{jjii}=-\lambda_{i}\lambda_{j}(\lambda_{i}-\lambda_{j}),\ \ h_{iikl}- h_{iilk}=0, \ \ k\neq l.
\end{equation}

In 2015, Cheng and Peng \cite{CP} obtain the rigidity theorems on complete self-shrinkers without the assumption on polynomial volume growth. Motivated by the above work, we obtain the following result on self-expanders.

\begin{theorem}\label{theorem 3.1} (Theorem \ref{theorem 1.1})
Let $x:M^{n}\rightarrow \mathbb{R}^{n+1}_{1}$
be a complete space-like self-expander with the constant squared norm $S$ of the second fundamental form. If $\inf H^{2}>0$, then $S=1$ and $x(M^{n})$ is the hyperbolic space $\mathbb{H}^{n}(\sqrt{n})$ or the hyperbolic cylinder $\mathbb{H}^{k}(\sqrt{k})\times\mathbb{R}^{n-k}, \ \ 1\leq k\leq n-1$.
\end{theorem}
\begin{proof}
Since $\inf H^{2}>0$, without loss of generality, we can assume that $\inf H>0$.
For the constant $S$, from\eqref{2.1-13} and the Lemma \ref{lemma 2.4}, we infer that
\begin{equation}\label{3.1-3}
\sum_{i,j,k}h_{ijk}^{2}+S(S-1)=0.
\end{equation}
and the Bakry-Emery Ricci curvature of $x:M^{n}\rightarrow \mathbb{R}^{n+1}_{1}$ is bounded from below.
By the maximum principle
of Omori-Yau type concerning the operator $\mathcal{L}$ to the function $-H$, there exists a sequence $\{p_{t}\} \in M^{2}$ such that
$$0\leq\lim_{t\rightarrow\infty}\mathcal{L}H(p_{t})=\inf H(S-1).$$
It is obvious that $S\geq1$, and then $S=1$ and $\sum_{i,j,k}h_{ijk}^{2}=0$ from \eqref{3.1-3}. Namely, the second fundamental form of $x(M^{n})$ is parallel and $H$ is constant. It follows from \eqref{2.1-10} that
$\lambda_{i}(H\lambda_{i}-1)=0, \ 1\leq i\leq n$.
Then we infer that $\lambda_{i}=0$ or $\frac{1}{H}$ for any $1\leq i\leq n$, which indicates
that the number of distinct principal curvatures is at most two.
If $x(M^{n})$ has only one principal curvature, that is, $x(M^{n})$ is totally umbilic, then $x(M^{n})$ is the hyperbolic space $\mathbb{H}^{n}(\sqrt{n})$.
If $x(M^{n})$ has just two distinct constant principal curvatures, by the congruence
theorem of Abe, Koike and Yamaguchi \cite{AKY}, $x(M^{n})$ is
the hyperbolic cylinder $\mathbb{H}^{k}(\sqrt{k})\times\mathbb{R}^{n-k}, \ 1\leq k\leq n-1$.
\end{proof}

In order to prove the main theorem of this paper, we need the following theorem.

\begin{theorem}\label{theorem 3.2}
Let $x:M^{2}\rightarrow \mathbb{R}^{3}_{1}$
be a complete space-like self-expander with the non-zero constant squared norm $S$ of the second fundamental form, then $\inf H^{2}>0$.
\end{theorem}
\begin{proof}
We will use the proof by contradiction.
Suppose $\inf H^{2}=0$, there exists a sequence $\{p_{t}\}$ in $M^{2}$ such that
\begin{equation*}
\lim_{t\rightarrow\infty} H^{2}(p_{t})=\inf H^{2}=\bar H^2=0.
\end{equation*}
By \eqref{2.1-13}, \eqref{2.1-14} and $S$ being constant, we know that
$\{h_{ij}(p_{t})\}$,  $\{h_{ijk}(p_{t})\}$ and $\{h_{ijkl}(p_{t})\}$ are bounded sequences, one can assume
$$\lim_{t\rightarrow\infty}h_{ij}(p_{t})=\bar h_{ij}=\bar \lambda_{i}\delta_{ij}, \ \  \lim_{t\rightarrow\infty}h_{ijk}(p_{t})=\bar h_{ijk}, \ \ \lim_{t\rightarrow\infty}h_{ijkl}(p_{t})=\bar h_{ijkl}, \ \ i, j, k, l=1, 2.$$
It follows from \eqref{2.1-10} that
\begin{equation}\label{3.1-4}
\bar h_{11i}+\bar h_{22i}=-\bar \lambda_{i}\lim_{t\rightarrow\infty} \langle X, e_{i} \rangle(p_{t}), \ \ i=1, 2
\end{equation}
and
\begin{equation}\label{3.1-5}
\bar h_{11ij}+\bar h_{22ij}=-\sum_{k}\bar h_{ijk}\lim_{t\rightarrow\infty} \langle X,e_{k} \rangle(p_{t})-\bar \lambda_{i} \delta_{ij}, \ \ i, j=1, 2.
\end{equation}
By use of \eqref{3.1-1}, we obtain
\begin{equation}\label{3.1-6}
\bar \lambda_{1}\bar h_{11k}+\bar \lambda_{2}\bar h_{22k}=0
\end{equation}
and
\begin{equation}\label{3.1-7}
\bar \lambda_{1}\bar h_{11kl}+\bar \lambda_{2}\bar h_{22kl}=-\sum_{ij}\bar h_{ijk}\bar h_{ijl}, \ \ k, l=1, 2.
\end{equation}
\eqref{3.1-2} yields
\begin{equation}\label{3.1-8}
\bar h_{1112}-\bar h_{1121}=0, \ \ \bar h_{2221}-\bar h_{2212}=0.
\end{equation}
Since $\bar H^{2}=0$ and $S$ is non-zero constant, we draw that
$$\bar \lambda_{1}=-\bar \lambda_{2}\neq 0, \ \ S=2\bar \lambda^{2}_{1}.$$
Then \eqref{3.1-4} and \eqref{3.1-6} imply
\begin{equation}\label{3.1-9}
\bar h_{11k}=\bar h_{22k}, \ \ \lim_{t\rightarrow\infty}\langle X, e_{k}\rangle(p_{t})=-\frac{2\bar h_{11k}}{\bar \lambda_{k}},\ \ k=1,2.
\end{equation}
It follows from \eqref{3.1-5}, \eqref{3.1-7} and \eqref{3.1-9} that
\begin{equation*}
\begin{aligned}
&\bar h_{1111}+\bar h_{2211}=\frac{2(\bar h^{2}_{111}-\bar h^{2}_{112})}{\bar \lambda_{1}}-\bar \lambda_{1}, \ \
\bar h_{1122}+\bar h_{2222}=\frac{2(\bar h^{2}_{111}-\bar h^{2}_{112})}{\bar \lambda_{1}}+\bar \lambda_{1},\\
&\bar h_{1112}+\bar h_{2212}=0
\end{aligned}
\end{equation*}
and
\begin{equation*}
\begin{aligned}
&\bar \lambda_{1}(\bar h_{1111}-\bar h_{2211})=-2\bar h^{2}_{111}-2\bar h^{2}_{112}, \ \
\bar \lambda_{1}(\bar h_{1122}-\bar h_{2222})=-2\bar h^{2}_{111}-2\bar h^{2}_{112},\\
&\bar \lambda_{1}(\bar h_{1112}-\bar h_{2212})=-4\bar h_{111}\bar h_{112}.
\end{aligned}
\end{equation*}
Thus,
\begin{equation}\label{3.1-10}
\begin{cases}
\begin{aligned}
&\bar h_{1111}=-\frac{2\bar h^{2}_{112}}{\bar \lambda_{1}}-\frac{\bar \lambda_{1}}{2}, \ \
\bar h_{2211}=\frac{2\bar h^{2}_{111}}{\bar \lambda_{1}}-\frac{\bar \lambda_{1}}{2},\ \
\bar h_{1122}=-\frac{2\bar h^{2}_{112}}{\bar \lambda_{1}}+\frac{\bar \lambda_{1}}{2}, \\
&\bar h_{2222}=\frac{2\bar h^{2}_{111}}{\bar \lambda_{1}}+\frac{\bar \lambda_{1}}{2}, \ \
\bar h_{1112}=-\frac{2\bar h_{111}\bar h_{112}}{\bar \lambda_{1}}, \ \
\bar h_{2212}=\frac{2\bar h_{111}\bar h_{112}}{\bar \lambda_{1}}.
\end{aligned}
\end{cases}
\end{equation}
It follows from \eqref{2.1-13} that
\begin{equation}\label{3.1-11}
\sum_{i,j,k}h_{ijk}^{2}=S(1-S), \ \ \sum_{i,j,k}\bar h_{ijk}\bar h_{ijkl}=0, \ \ l=1, 2.
\end{equation}
Specifically,
\begin{equation*}
\begin{aligned}
&\bar h_{111}\bar h_{1111}+\bar h_{222}\bar h_{2221}+3\bar h_{112}\bar h_{1121}+3\bar h_{221}\bar h_{2211}=0, \\
&\bar h_{111}\bar h_{1112}+\bar h_{222}\bar h_{2222}+3\bar h_{112}\bar h_{1122}+3\bar h_{221}\bar h_{2212}=0.
\end{aligned}
\end{equation*}
Then by \eqref{3.1-8} and \eqref{3.1-9}, we conclude that
\begin{equation}\label{3.1-12}
\begin{aligned}
&\bar h_{111}\bar h_{1111}+\bar h_{112}\bar h_{2212}+3\bar h_{112}\bar h_{1112}+3\bar h_{111}\bar h_{2211}=0, \\
&\bar h_{111}\bar h_{1112}+\bar h_{112}\bar h_{2222}+3\bar h_{112}\bar h_{1122}+3\bar h_{111}\bar h_{2212}=0.
\end{aligned}
\end{equation}
Combining \eqref{3.1-10} with \eqref{3.1-12}, we infer
\begin{equation}\label{3.1-13}
\bar h_{111}(3\bar h^{2}_{111}-3\bar h^{2}_{112}-\bar \lambda^{2}_{1})=0, \ \
\bar h_{112}(3\bar h^{2}_{111}-3\bar h^{2}_{112}+\bar \lambda^{2}_{1})=0.
\end{equation}
If $\bar h_{111}\bar h_{112}\neq0$, \eqref{3.1-13} implies
that $\bar \lambda^{2}_{1}=S=0$.
It is a contradiction.
If $\bar h_{111}\bar h_{112}=0$, we can also obtain contradictions. In fact,
we can declare that $\sum_{i,j,k}\bar h^{2}_{ijk}\neq0$. Otherwise, we know that $\bar h_{ijkl}=0$ for $i,j,k,l=1,2$ from \eqref{2.1-14}.
Then by \eqref{3.1-10}, we infer $\bar \lambda_{1}=0$. It is a contradiction.
Assuming  $\bar h_{111}\neq0, \ \bar h_{112}=0$. From the first equation of \eqref{3.1-11} and \eqref{3.1-13}, we draw
$$\bar h^{2}_{111}=\frac{1}{4}S(1-S), \ \ \bar h^{2}_{111}=\frac{1}{3}\bar \lambda^{2}_{1}=\frac{1}{6}S
.$$
Then
\begin{equation}\label{3.1-14}
S=\frac{1}{3}, \ \ \bar \lambda^{2}_{1}=\frac{1}{6}, \ \ \bar h^{2}_{111}=\frac{1}{18}.
\end{equation}
Consequently
the following relationship
\begin{equation}\label{3.1-15}
\bar h_{1111}=-\frac{\bar \lambda_{1}}{2}, \ \
\bar h_{2211}=\frac{\bar \lambda_{1}}{6},\ \
\bar h_{1122}=\frac{\bar \lambda_{1}}{2}, \ \
\bar h_{2222}=\frac{7\bar \lambda_{1}}{6}, \ \
\bar h_{1112}=\bar h_{2212}=0
\end{equation}
can be derived by the simple calculation from \eqref{3.1-10}.

\noindent
It follows from \eqref{2.1-14}, \eqref{3.1-14} and \eqref{3.1-15} that
\begin{equation*}
\sum_{i,j,k,l}(\bar h_{ijkl})^{2}=\bar h^{2}_{1111}
+\bar h^{2}_{2222}+3\bar h^{2}_{1122}+3\bar h^{2}_{2211}+4\bar h^{2}_{1112}+4\bar h^{2}_{2212}=\frac{11}{27}
\end{equation*}
and
\begin{equation*}
\begin{aligned}
\sum_{i,j,k,l}(\bar h_{ijkl})^{2}
=&(2-S)\sum_{i,j,k}(\bar h_{ijk})^{2}
+6\sum_{i,j,k}\bar \lambda_{i}\bar \lambda_{j}\bar h^{2}_{ijk}-3\sum_{i,j,k}\bar \lambda^{2}_{k}\bar h^{2}_{ijk} \\
=&4(2-S)\bar h^{2}_{111}-12\bar \lambda^{2}_{1}\bar h^{2}_{111}
=\frac{7}{27}.
\end{aligned}
\end{equation*}
This is impossible. Assuming $\bar h_{111}=0, \ \bar h_{112}\neq0$. Using the similar methods to also obtain contradictions. The proof of the Theorem \ref{theorem 3.2} is thus finished.
\end{proof}

\vskip3mm
\noindent
{\it Proof of Theorem \ref{theorem 1.2}}.
If $S=0$, we know that $x: M^{2}\to \mathbb{R}^{3}_{1}$ is a space-like affine plane $\mathbb{R}^{2}$.
If $S\neq0$, it is obvious that
either $x(M^{2})$ is $\mathbb{H}^{1}(1)\times\mathbb {R}^{1}$ or $\mathbb{H}^{2}(\sqrt{2})$ from the Theorem \ref{theorem 3.1} and the Theorem \ref{theorem 3.2}.
\begin{flushright}
$\square$
\end{flushright}

\noindent{\bf Acknowledgements}
The first author was partially supported by the China Postdoctoral Science Foundation Grant No.2022M711074. The second author was partly supported by grant No.12171164 of NSFC, GDUPS (2018), Guangdong Natural Science Foundation Grant No.2023A1515010510.

\end{document}